\documentclass[11pt, reqno]{amsart}
\usepackage{amsmath}
\usepackage{amsfonts}
\usepackage{amsthm}
\usepackage{amssymb}
\usepackage{textcomp}
\usepackage{graphicx}
\usepackage{color}
\usepackage{mathrsfs}
\usepackage{enumerate}

\usepackage[bookmarks, pdftitle={Uniform Dilations in Higher Dimensions}, pdfauthor={Thai Hoang Le and Mike Kelly }, colorlinks=true, linkcolor=black, citecolor=black, urlcolor=black]{hyperref}

      \newcommand{\ds}{\displaystyle}

      \newcommand{\Rn}{\mathbb{R}^{N}}
      \newcommand{\Rl}{\mathbb{R}^{L}}
      
      \newcommand{\Zn}{\mathbb{Z}^N}
      \newcommand{\Tn}{\mathbb{T}^N}
      \newcommand{\Tl}{\mathbb{T}^L}
      \newcommand{\Z}{\mathbb{Z}}
      \newcommand{\Zl}{\mathbb{Z}^{L}}
      \newcommand{\R}{\mathbb{R}}
      \newcommand{\Q}{\mathbb{Q}}
      \newcommand{\Qn}{\mathbb{Q^{N}}}
      \newcommand{\C}{\mathbb{C}}
      \newcommand{\T}{\mathbb{T}}
      \newcommand{\lb}{\left\lbrace}
      \newcommand{\rb}{\right\rbrace}

      \newcommand{\dsum}{\displaystyle\sum}
      
      \newcommand{\dprod}{\displaystyle\prod}
      \newcommand{\ra}{\rightarrow}
      
      \newcommand{\bxi}{\boldsymbol{\xi}}
      \newcommand{\balpha}{\boldsymbol{\alpha}}
      
      \newcommand{\btheta}{\boldsymbol{\theta}}
      \newcommand{\bomega}{\boldsymbol{\omega}}
      \newcommand{\ba}{\mathbf{a}}
      \newcommand{\bb}{\mathbf{b}}
      \newcommand{\be}{\mathbf{e}}
      \newcommand{\bx}{\mathbf{x}}
      \newcommand{\by}{\mathbf{y}}
      \newcommand{\bt}{\mathbf{t}}
      \newcommand{\bm}{\mathbf{m}}
      \newcommand{\bq}{\mathbf{q}}
      \newcommand{\bw}{\mathbf{w}}
      \newcommand{\bv}{\mathbf{v}}
      \newcommand{\bC}{\mathbf{C}}
      \newcommand{\bA}{\mathbf{A}}
      \newcommand{\bB}{\mathbf{B}}
      \newcommand{\bzero}{\mathbf{0}} 
      \newcommand{\mT}{\mathcal{T}}

      \newcommand{\mM}{\mathrm{M}}
      
      \newcommand{\ms}{\mathtt{s}}
      \newcommand{\mS}{\mathtt{S}}
      
      \newcommand{\ep}{\epsilon}

\theoremstyle{remark}
	\newtheorem{remark}{Remark}
	\newtheorem{remarks}{Remarks}
	\newtheorem{example}{Example}

\theoremstyle{plain}
  \newtheorem{theorem}{Theorem}

  \newtheorem{proposition}{Proposition}
  
  \newtheorem{lemma}{Lemma}
  
  \newtheorem{corollary}{Corollary}

\title{Uniform Dilations in Higher Dimensions}
\author{Michael Kelly}
\email{mkelly@math.utexas.edu}
\author{Th\'{a}i Ho\`{a}ng L\^{e}}
\email{leth@math.utexas.edu}
\address{The University of Texas at Austin \\ 1 University Station C1200\\ Austin, TX, USA 78712}
\date{\today}

\begin{document}
\begin{abstract}
A theorem of Glasner says that if $X$ is an infinite subset of the torus $\T$, then for any $\epsilon>0$, there exists an integer $n$ such that the dilation $nX=\{nx: x \in \T \}$ is $\ep$-dense (i.e, it intersects any interval of length $2\ep$ in $\T$). Alon and Peres provided a general framework for this problem, and showed quantitatively that one can restrict the dilation to be of the form $f(n)X$ where $f \in \Z[x]$ is not constant. Building upon the work of Alon and Peres, we study this phenomenon in higher dimensions. Let $\bA(x)$ be an $L \times N$ matrix whose entries are in $\Z[x]$, and $X$ be an infinite subset of $\T^N$. Contrarily to the case $N=L=1$, it's not always true that there is an integer $n$ such that $\bA(n)X$ is $\ep$-dense in a translate of a subtorus of $\T^{L}$. We give a necessary and sufficient condition for matrices $\bA$ for which this is true. We also prove an effective version of the result.
\end{abstract}
\maketitle
\pagestyle{myheadings}

						      \section{Introduction}

Let $\T=\R/\Z$. A subset $X \subset \T$ is called  $\ep$-\textit{dense} in $\T$ if it intersects every interval of length $2\ep$ in $\T$. 
A \textit{dilation} of $X$ is a set of the form $nX=\lb nx :x \in X \rb \subset \T$. The following theorem of Glasner \cite{MR531259} is the basis for our investigation.

\newtheorem{glasner}{Theorem}
\renewcommand*{\theglasner}{\Roman{glasner}}			
			\begin{glasner}[Glasner] \label{Glasner}
					Let $X$ be an infinite subset of $\T$ and $\epsilon>0$, then there exists a positive integer $n$ such that the dilation $nX$ is $\epsilon$-dense in $\T$.
			\end{glasner}
Theorem \ref{Glasner} can be made effective in the sense that every sufficiently large subset $X$ has an $\ep$-dense dilation of the form $nX$ for some positive integer $n$, and `sufficiently large' can be quantified. The first result in this direction was obtained by Berend and Peres in \cite{MR1200973}. 
Given $\ep>0$, let $k(\ep)$ be the minimal integer $k$ such that for any set $X\subset \T$ of cardinality at least $k$, some dilation $n X$ is $\ep$-dense in $\T$. Berend and Peres showed that
					      \begin{equation} \label{glasnerEff0}
						      c/\ep^{2}\leq k(\ep)\leq (c_{1}/\ep)^{c_{2}/\ep}
					      \end{equation}
					where $c,c_{1},c_{2}$ are absolute constants.

The question of determining the correct order of magnitude of $k(\ep)$ was further studied in depth by Alon and Peres \cite{MR1143662}, who gave the bound  
\begin{equation} \label{glasnerEff}
k(\ep) \ll_{\delta} \left( \frac{1}{\epsilon} \right)^{2+\delta}
\end{equation}
for any $\delta>0$. This is almost best possible in view of (\ref{glasnerEff0}). Actually, they gave a more precise bound
\begin{equation} \label{glasnerEff2}
				k(\ep) \ll \left(\dfrac{1}{\ep}\right)^{2+\frac{3}{\log\log (1/\ep)}}.
\end{equation}

In \cite{MR1143662}, Alon and Peres provided two different approaches to this problem. On the one hand, the probabilistic approach gives more information about the dilation, such as its discrepancy. On the other hand, the second approach, using harmonic analysis, is particular suited when one is interested 
in dilating the set $X$ by a sequence of arithmetic nature, such as the primes or the squares. They proved

\newtheorem{alonPeres6.3}[glasner]{Theorem}
 			\begin{alonPeres6.3}[Alon-Peres] \label{alonPeres}
 						\begin{enumerate}[(i)]
 								\item \label{ap1} For any $\delta>0$, every set $X$ in $\T$ of cardinality 
 											\begin{equation*}
 													k \gg_{\delta} \dfrac{1}{\ep^{2+\delta}},
 											\end{equation*}
 											has an $\ep$-dense dilation $pX$ with $p$ prime.
 								\item \label{ap2} Let $f$ be a polynomial of degree $L>1$ with integer coefficients and let $\delta>0$. Then any set $X$ in $\T$ of cardinality 
 											\begin{equation*}
 													k \gg_{\delta, f} \left( \dfrac{1}{\ep} \right)^{2L+\delta},
 											\end{equation*}
 										has an $\ep$-dense dilation of the form $f(n)X$, for some $n\in\Z$.
 						\end{enumerate}
 			\end{alonPeres6.3}
 			
It is shown in \cite{MR1452815} that in part (ii) of the above theorem there is an $\ep$-dense dilation of the form $f(p)X$ where $p$ is a prime number.

In this paper we investigate high dimensional analogues of Glasner's theorem and the above results of Alon and Peres using Alon-Peres' harmonic analysis approach. One problem that comes to mind is that of determining the natural analogue of ``dilating by $n$'' in the one-dimensional case. Any continuous endomorphism of $\T$ is represented this way, so we may regard the dilation as the action by a continuous endomorphism. When considering higher dimensional generalizations of the above theorems we need not restrict ourselves from maps of a torus into itself. We will instead consider maps between tori of possibly different dimension. A continuous homomorphism between $\Tn$ and $\Tl$ is represented by left multiplication of an $L\times N$ matrix with entries in $\Z$. This will be our analogue of dilation. We say that a subset of $\Tl$ is $\ep$-\textit{dense} in $\Tl$ if it intersects any box of side length $2\ep$.

Our first theorem is a high dimensional analogue of Glasner's theorem.
\begin{theorem}\label{HDGlasner}
						For any $\ep>0$ and any infinite subset $X\subset \Tn$ there exists a continuous homomorphism $T:\Tn\ra\Tl$ such that $TX$ is $\ep$-dense in $\Tl$.
\end{theorem}
The proof of this result is similar to the proof of (\ref{glasnerEff}). Our main investigation, however, is an analogue of the fact that if $X \subset \T$ is infinite, then there is a dilation of the form $f(n)X$ that is $\ep$-dense, where $f(x)$ is a non-constant polynomial with integral coefficients. Let us introduce the set-up to this problem and lay out some of the complications that arise when moving to high dimensions. In this paper, a \textit{subtorus} of $\Tn$ is defined to be a non-trivial closed and connected Lie subgroup.

Let $\bA(x)\in\mM_{L\times N}(\Z[x])$ be non-constant and let $D$ be the positive integer representing the largest of the degrees of the entries of $\bA(x)$. Then there are $A_{0},...,A_{D}\in\mM_{L\times N}(\Z)$ such that 
		\[
		      \bA(x)=A_{0}+x A_{1}+\cdots +x^{D}A_{D}=A_{0}+\bA_{\ast}(x)
		\]
where $\bA_{\ast}(x)$ is the \textit{non-constant part} of $\bA(x)$. We wish to consider dilations of subsets $X\subset \Tn$ of the form $\bA(n)X$.

Simple examples show that, unlike Theorem \ref{HDGlasner}, there are configurations of $\bA(x)$ and $X$ for which $\bA(n)X$ is never $\ep$-dense in the full torus. Take, for instance, 
$\bA(n)=\begin{pmatrix} 
         n & 0\\
         0 & n
        \end{pmatrix}$
and $X$ to live in a proper subtorus, then $A(n)X$ is also in the same subtorus, for every $n$. Furthermore, if we take $X$ to be in a translate of a subtorus, then $A(n)X$ is also in a translate of a subtorus (where the translate depends of $n$). So the best one can hope for in this situation is to achieve an $\ep$-dense dilation in a \textit{translate of a subtorus}. Before stating our results, we give some examples to show that even this restriction is not always achieved.
\begin{example} \label{ex1}
If $\bA(n)=\begin{pmatrix} 
         n & 0\\
         0 & 0
        \end{pmatrix}$ and $X=\{(0,x): |x| \leq 1/4 \}$, then there is no value of $n$ such that $A(n)X$ is $1/4$-dense in a translate of a subtorus. Basically, this is because the matrix $\bA_{\ast}$ is degenerate in a sense so that $\bA(n)X$ doesn't ``move $X$ around.''
\end{example}

\begin{example} \label{ex2}
If $\bA(n)=\begin{pmatrix} 
         n & 0\\
         0 & n+1
         \end{pmatrix}$ and $X=\{(1/j, 1/j): j=1,2,\ldots \}$, then clearly $\bA(n)X$ is not $1/4$-dense in any translate of the diagonal. On the other hand, one can show that for any $n$, for any subtorus $\mT$ of $\T^2$ that is different from the diagonal, $\bA(n)X$ is not $\epsilon$-dense in any translate of $\mT$ (since the set of dot products of elements of $\bA(n)X$ with $(-1 \quad 1)$ has only one accumulation point). The reason of such a failure can be attributed to the lack of a compromise between the constant part and the non-constant part of $\bA$.
\end{example}

Our main result says that the only obstructions to $\ep$-dense dilations are the ones described in Examples \ref{ex1} and \ref{ex2}.

\begin{theorem}\label{hoangMike}
					 Let $\bA(x)\in\mM_{L\times N}(\Z[x])$. The following are equivalent:
					 
					 \begin{enumerate}
					\item \label{epdense} For any infinite subset $X\subset\Tn$ there exists a subtorus $\mT=\mT(X,\bA)$ of $\Tl$ such that for any $\ep>0$ there exists an integer $n$ such that $\bA(n)X=\lb \bA(n)\bx: \bx\in X \rb$ is 
					$\ep$-dense in a translate of $\mT$.
					  \item \label{cond}
					   		\begin{enumerate}[(a)]
					     \item \label{cond1} The columns of $\bA_{\ast}(x)$ are $\Q$-linearly independent, and
					     \item \label{cond2} If there are $\bv\in\Q^{L}$ and $\bw\in\Q^{N}$ satisfying 
								\begin{equation}
								      \bv\cdot A_{d}\bw=0 \;\;\;\;\text{ for each }\;d=1,...,D,
								\end{equation}
						  then $\bv\cdot A_{0}\bw=0$.
					    \end{enumerate}
					\end{enumerate}
\end{theorem}
		
\begin{remarks}
 \noindent \begin{itemize}
  \item Theorem \ref{hoangMike} shows one how to construct matrices $\bA(n)$ such that the conclusion (\ref{epdense}) holds. The condition (\ref{cond1}) tells us how to choose the non-constant part $\bA_{\ast}(n)$, and the condition (\ref{cond2}) tells us that the constant part $A_0$ has to behave accordingly. 
  \item In the case $N=L=1$, (\ref{cond}) is automatically satisfied if $\bA$ is not constant, which explains why in Theorem \ref{alonPeres} (\ref{ap2}) we can take $f$ to be any non-constant polynomial.
  \item If we replace $\Q$ with $\C$ in (\ref{cond2}), then by Hilbert's Nullstellensatz, it would imply that $A_0$ is a linear combination of $A_1, \ldots, A_D$. It would be interesting to construct examples of $\bA$ satisfying (\ref{cond2}) without $A_0$ being a linear combination of $A_1, \ldots, A_D$.
 \end{itemize}
\end{remarks}

We also prove an effective form of this result. Define $k(\ep;L,N,\bA)$ to be the largest integer $k$ such that there exist $k$ distinct points $X=\{\bx_{1},...,\bx_{k}\}\subset\Tn$ such that $\bA(n)X=\lb \bA(n)\bx_{1},...,\bA(n)\bx_{k} \rb$ is not $\ep$-dense in any translate of any subtorus for any $n=1,2,3,...$.
 
\begin{theorem}\label{polynomialEffective}
Let $\bA(x)$ be of degree at most $D$ and satisfy (\ref{cond1}) and (\ref{cond2}) from Theorem \ref{hoangMike}. Then there are constants $c_{1}(N,L,D)$ and $c_{2}(N,L,D)$ such that 
						\begin{equation}
					k(\ep;L,N,\bA)\ll_{N,L,D} \|\bA_{\ast}\|_{\infty}^{c_{1}(N,L,D)}\left( \dfrac{1}{\ep} \right)^{c_{2}(N,L,D)}.
						\end{equation}
where $\|\bA_{\ast}\|_{\infty}$ is the max of the heights\footnote{\textup{Recall that the height of a polynomial is the maximum of the absolute values of its coefficients.}} of the entries of $\bA_{\ast}$.
\end{theorem}
\begin{remark}
Theorem \ref{hoangMike} would be a mere consequence of Theorem \ref{polynomialEffective}, if not for the fact that the subtorus $\mT$ is independent of $\ep$ in the conclusion of Theorem \ref{hoangMike}.
\end{remark}
The exponents $c_1$ and $c_2$ can be given explicitly. We do not try to find the best possible exponents, since these are not known even in the case $N=L=1$, though our values can certainly be improved. Finally, we remark that it is straightforward to prove a version of Theorem \ref{polynomialEffective} in the spirit of \cite{MR1452815}, with bounds of the same quality, for dilations of the form $\bA(p)X$ where $p$ is prime. Indeed, the proof would proceed exactly the same way, albeit with an appropriate modification of Lemma \ref{hua}. We leave the details to the interested reader.

The paper is organized as follows. In Section \ref{prelim} we gather some useful facts that we need in our proofs, including Alon-Peres' machinery. In Section \ref{infinite} we prove Theorem \ref{hoangMike}, and in Section \ref{finite} we prove Theorem \ref{polynomialEffective}. In Section \ref{glasnerSection} we prove (a variant of) a quantitative version of Theorem \ref{HDGlasner}. Finally, in Section \ref{applications} we discuss some applications of our results.\\

\noindent\textbf{Acknowledgements.} We would like to thank Professor Noga Alon for a discussion regarding Proposition \ref{count} and Professor Jeffrey Vaaler for helpful comments during our investigation and during the preparation of this paper.

						      \section{Notation and preliminaries} \label{prelim}

\subsection{Notation} Throughout this paper, we will use Vinogradov's symbols $\ll$ and $\gg$. For two quantities $A,B$, we write $A \ll B$, or $B \gg A$ if there is a positive constant $c$ such that $|A|\leq cB$. If the constant $c$ depends on another quantity $t$, then we indicate this dependence as $A \ll_{t} B$. The numbers $N,L,D$ are fixed throughout this paper, so dependence on these quantities is implicitly understood.

Given a vector $\bv$, we denote by $\| \bv \|_{\infty}$ its usual sup norm. Given a matrix $A$, let us denote by $\| A \|_{\infty}$ the maximal of the absolute values of its entries. Finally, for a matrix $\bA(x)=A_{0}+x A_{1}+\cdots +x^{D}A_{D}$ whose entries a polynomials in $x$, we define $\| \bA \|_{\infty}=\max \{ \|A_{d} \|_{\infty}: d=0,1, \ldots, D \}$. While we use the same symbol for slightly different objects, the use should be clear from the context.

For $x \in \R$, we denote by $\| x \|$ the distance from $x$ to the nearest integer. For $\bx=(x_1, \ldots, x_{\ell}) \in \R^{\ell}$, let $\| \bx \| = \max_{i=1, \ldots, \ell} \|x_i\|$. In other words, $\| \bx \|$ denotes the distance from $\bx$ to the nearest integer lattice point under $\|\cdot \|_{\infty}$. 

Throughout the paper, we always identify a point in a torus $\T^{\ell}$ with its unique representative in $[0,1)^{\ell}$. This point of view is important, since it enables us to define subtori in terms of equations.
\subsection{Preliminaries}
Let $\lb x_{1},...,x_{k} \rb$ be a set of $k$ distinct numbers in $\T$. Define
	  \begin{equation}\label{hm}
		h_{m}=\#\lb (i,j):1\leq i,j\leq k\;\text{ and }m(x_{i}-x_{j})\in\Z \rb	
	  \end{equation}
and $H_{m}=h_{1}+\cdots+h_{m}$. The quantities $h_i, H_m$ certainly depend on the sequence $\lb x_{1},...,x_{k} \rb$, but we always specify the sequence we are working with. The numbers $h_{m}$ and $H_{m}$ appear in several of the arguments in \cite{MR1143662} and they will make an appearance in the proof of our main results. We will need the following simple estimate:
\begin{proposition} \label{count}
 $H_{m}\leq km^{2}.$
\end{proposition}
\begin{proof}
 Observe that for fixed $i$ and $m$, there are at most $m$ values of $j$ such that $m(x_i-x_j) \in \Z$. Thus for fixed $i$, the number of couples $(j,m)$ 
 such that $m(x_i-x_j) \in \Z$ is at most $1+\cdots+M \leq M^2$. Summing this up over all $i$ gives the desired estimate. 
\end{proof}
\begin{remark}
Since we are not concerned with optimal exponents, this estimate will suffice for our purposes, but we note that it is shown in \cite{MR1143662} that the (essentially sharp) bound $H_{m}\ll_{\gamma} (mk)^{1+\gamma}$ holds for any $\gamma>0$. 
\end{remark}

\begin{corollary}\label{alonPeresCorollary2}
				If $\ms_{2},\ms_{3},...$ is a sequence of positive integers such that $\mS_{b}=\ms_{2}+\cdots+\ms_{b}\leq H_{b}$ and $\mS_{b}\leq k^{2}$, then 
					      \begin{equation}
						    \dsum_{b=2}^{\infty}\ms_{b}b^{-1/D} \ll_{D} k^{2-1/(2D)}.
					      \end{equation}
\end{corollary}
			\begin{proof} We follow the proof of a similar estimate in \cite{MR1143662}. For $b\geq \sqrt{k}$ use the  bound $\mS_{b}\leq k^{2}$ and if $b>\sqrt{k}$ use $\mS_{b}\leq H_{b}\ll kb^{2}$ so we have by summation by parts
					  \[
					    \dsum_{b=2}^{\infty}\mS_{b}\Big( b^{-1/D}-(b+1)^{-1/D} \Big)
					    \ll 
					    k^{2}k^{-1/(2D)}+k\dsum_{b=2}^{\sqrt{k}}b^{2} b^{-1/D-1}. 
					  \]
				      But \[ \dsum_{b=2}^{\sqrt{k}}b^{1-1/D}\ll_{D} k^{1-1/(2D)}.  \]
			\end{proof}

The following Lemma is a high dimensional analogue of an inequality used in the several of the results in \cite{MR1143662}. It may be regarded as a general principle 
which connects the lack of $\ep$-denseness to exponential sums.

\begin{proposition}\label{mainInequality}
	      Let $A(1),A(2),...$ be a sequence of linear transformations taking $\Tn$ to $\T^{\ell}$ and assume $X=\lb\bx_{1},...,\bx_{k}\rb$ is a subset of $\Tn$ of cardinality $k$ such that $\bA(n)X$ is not $\ep$-dense in $\T^{\ell}$ for any $n\in\Z$. Then for any $\ep>0$ there is an integer $0 \leq M \ll_{\ell} \ep^{-1}$ such that 
	      \begin{equation}\label{main}
		     k^{2}\ll_{\ell} \dfrac{1}{\ep^{\ell}}\underset{\bm\in\Z^\ell}{\dsum_{0<\|\bm\|_{\infty}\leq M}} \dsum_{i=1}^{k}\dsum_{j=1}^{k} \lim_{R \rightarrow \infty} \dfrac{1}{R}\dsum_{r=1}^{R} e_{\bm}\Big(\bA(r)(\bx_{i}-\bx_{j})\Big)
	      \end{equation}
	      where $e_{\bm}(\bt)=\exp(2\pi i \bm\cdot\bt)$.
\end{proposition}
Alon-Peres proved the one-dimensional version of Lemma \ref{mainInequality} using a classical result of Denjoy and Carleman, and obtained the same inequality with
$M \ll (1/\ep)\log^2 (1/ \ep)$. Their method can be extended in a straightforward manner to higher dimensions. As pointed out to us by Vaaler, one could as well use the machinery developed by Barton-Montgomery-Vaaler \cite{bmv} to improve this to $M \ll 1/\ep$. We will follow the latter approach in our proof of Proposition \ref{mainInequality} since it gives us a cleaner value for $M$, though this is inconsequential. Indeed, even in the case $N=L=1$, this improved value of $M$ does not lead to any improvement on Alon-Peres' bound (\ref{glasnerEff2}).

We first recall the following consequence of \cite[Corollary 2]{bmv}:
\begin{lemma} \label{multi-montgomery}
Let $0< \ep \leq 1/2$. Let $\bxi_1, \ldots, \bxi_{k} \in \R^{\ell}$ be such that $\| \bxi_{i} \| \geq \epsilon$ for any $i=1, \ldots, \ell$. Then we have
\[
\dfrac{k}{3} \leq \sum_{ \substack{\bm \in \Z^{\ell} \\ 0<\|\bm\|_{\infty}\leq \left[ \frac{\ell}{\ep} \right] } } \left| \sum_{i=1}^{k} e_{\bm}(\bxi_{i}) \right|
\]
\end{lemma}
\begin{proof}[Proof of Lemma \ref{mainInequality}] For any $r$, since $\bA(r)X$ is not $\ep$-dense in $\T^{\ell}$, there exists $\balpha_{r} \in \R^{\ell}$ such that $\| \balpha_{r} - \bA(r) \bx_i \| \geq \ep$ for any $i=1, \ldots, k$.
Let $M=\left[ \frac{\ell}{\ep} \right]$. By Lemma \ref{multi-montgomery}, we have
\[
\dfrac{k}{3} \leq \sum_{ \substack{\bm \in \Z^{\ell} \\ 0<\|\bm\|_{\infty}\leq M } } \left| \sum_{i=1}^{k} e_{\bm}(  \balpha_{r} - \bA(r) \bx_i) \right|
\]
By Cauchy-Schwarz, we have
\begin{eqnarray*}
k^2 &\ll_{\ell}& M^{\ell} \sum_{ \substack{\bm \in \Z^{\ell} \\ 0<\|\bm\|_{\infty}\leq M } } \left| \sum_{i=1}^{k} e_{\bm}(  \balpha_{r} - \bA(r) \bx_i) \right|^2 \\
&\ll_{\ell}& \dfrac{1}{\ep^{\ell}} \sum_{ \substack{\bm \in \Z^{\ell} \\ 0<\|\bm\|_{\infty}\leq M } } \sum_{i=1}^{k} \sum_{i=1}^{k} e_{\bm}(  \bA(r) (\bx_i-\bx_j))
\end{eqnarray*}
This is true for any $r$ so by taking the average of the right hand side over $1 \leq r \leq R$, we have
\[
k^{2}\ll_{\ell} \dfrac{1}{\ep^{\ell}}\underset{\bm\in\Z^\ell}{\dsum_{0<\|\bm\|_{\infty}\leq M}} \dsum_{i=1}^{k}\dsum_{j=1}^{k} \dfrac{1}{R}\dsum_{r=1}^{R} e_{\bm}\Big(\bA(r)(\bx_{i}-\bx_{j})\Big)
\]
Letting $R \rightarrow \infty$ we have the desired inequality.
\end{proof}
We also recall the following classical estimate due to Hua \cite{chen, necaev}:

\begin{lemma}[Hua] \label{hua}
		  Suppose $f(x)=a_{d}x^d +\cdots a_{1}x+a_{0}\in\Z[x]$ and $q$ is a positive integer such that $\gcd(a_{1},...,a_{d},q)=1$. Then
			      \[
				  \left| \dsum_{r=1}^{q}e^{2\pi i f(r)/q}\right| \ll_{d} q^{1-1/d} .
			      \]
\end{lemma}

						  \section{The infinite version} \label{infinite}

Of the two implications, the implication (\ref{epdense}) $\Rightarrow$ (\ref{cond}) is the more difficult so let us begin by quickly proving the implication (\ref{cond}) $\Rightarrow$ (\ref{epdense}). We will need the following lemma in the proof of the necessity of (\ref{cond2}). The assertion of the lemma is that by taking the dot product with a vector $\bv$, an $\ep$-dense subset of a torus becomes an $\tilde{\ep}$-dense set in $\T$ where $\tilde{\ep}$ is comparable to $\ep$, as long as $\bv$ is not orthogonal to the original torus.

\begin{lemma}\label{oneTorusLemma}
Let $\ep>0$, $\bb\in\Rl$, $V$ a proper subspace of $\Rl$, $\bv \in \Zl, \bv \not \in V^{\perp}$, and
\[ X \subset S=\lb \bb+\bx+\Zl: \bx\in V \rb\subset \Tl.\] 
If $X$ is $\ep$-dense in $S$, then $\lb \bv\cdot\bx+\Z:\bx\in X \rb$ is $L\|\bv\|_{\infty}\ep$-dense in $\T$. 
\end{lemma}
\begin{proof}
	    Let $t\in\T$. We want to find a $\bx\in X$ such that $\bv\cdot\bx$ is contained in an interval of length $2L\|\bv\|_{\infty}\ep$ in $\T$ centered at $t$. That is we wish to show the existence of an $\bx\in X$ such that $\|\bv\cdot\bx-t\|\leq L \|\bv\|_{\infty}\ep$.
	    
Since $\bv \not \in V^{\perp}$ we may write $t=\bv\cdot \ba$ for some $\ba\in V$. And since $X$ is $\ep$-dense in $S$ there exists an $\bx\in X\cap S$ and a $\bw\in \Zl$ such that $\|\bx-\ba-\bw\|_{\infty}\leq \ep$. But since $\bv\cdot\bw\in\Z$ we have
				  \[
					\|\bv\cdot\bx-t\|=\|\bv\cdot(\bx-\ba-\bw)\|\leq |\bv\cdot(\bx-\ba-\bw)|\leq L\|\bv\|_{\infty}\ep.
				  \]
\end{proof}

\begin{proof}[Proof of necessity of (\ref{cond1})]
Suppose, by way of contradiction, that the columns of $\bA_{\ast}(x)$ are not $\Q-$linearly independent. Then there is a nonzero $\bm\in\Q^{N}$ such that 
		\[
		      \bA_{\ast}\bm=0.
		\]
If \[X=\lb \bm/j : j=1,2,.... \rb,\] then $\bA(n)X=A_{0}X=\lb \bx_j= A_{0}\bm/j:j=1,2,.... \rb$ which is not $\ep$-dense in a translate of a subtorus for any sufficiently small $\ep>0$.
\end{proof}

\begin{proof}[Proof of necessity of (\ref{cond2})]
Suppose that there are vectors $\bv\in\Zl$ and $\bw\in\Zn$ such that 
			\[
			      \bv\cdot A_{d}\bw=0 \;\;\;\text{ for each }d=1,...,D
			\]
but $\bv\cdot A_{0}\bw=t\neq0$. In particular $\bv \neq \bzero$ and $\bw \neq \bzero$. Let $X=\lb\bw/j:j=1,2,...  \rb \subset \T^{N}$. Note that $X$ is an infinite set. It then follows that 
			 \begin{equation}\label{lineToContradict}
			      \bv\cdot\bA(n)\bx_{j}=t/j\searrow0 \;\;\; \text{ for each }n=1,2,...
			 \end{equation}

Suppose for a contradiction that there is a subtorus $S$ of $\Tl$ such that for any $\epsilon>0$, there exists $n$ such that $\bA(n)X$ is dense in a translate of $S$. Suppose $S$ is given by $S=\lb \bb+\ba+\Zl:\ba\in V \rb$ where $V$ is a proper subspace of $\Rl$ and $\bb\in\R^{L}$. Let $\ep>0$ be sufficiently small and suppose there is a subset $Y\subset X \cap S$, an integer $n$ such that $\bA(n)Y$ is $\ep$-dense in $S$. We have two possibilities:
\begin{itemize}
 \item If $\bv\in V^{\perp}$, then $\bv\cdot \bA(n) \by$ is a constant (namely $\bv \cdot \bb$) for any $\by \in Y$, which is not true in view of (\ref{lineToContradict}).
 \item If $\bv \not \in V^{\perp}$, then by Lemma \ref{oneTorusLemma} we have $\bv\cdot\bA(n)Y$ is $L\|\bv\|_{\infty}\ep$-dense in $\T$. Again, in view of (\ref{lineToContradict}), this is impossible if $\ep>0$ is sufficiently small.
\end{itemize}
\end{proof}

In the remainder of the paper we will say the \textit{rank} (\textit{corank}) of $\bA(x)$ is the rank of the $\Z$-module generated by the rows (columns) of
$\bA(x)$. First we describe briefly the ideas of the proof of the implication (\ref{cond}) $\Rightarrow$ (\ref{epdense}). 
Observe that we can't expect $\bA(n) X$ to be $\ep$-dense in the whole of $\Tl$ since there may be some linear dependencies between the rows of $\bA$. 
If $\bA(n) X$ fails to be $\ep$-dense in the ``natural" subtorus defined by these linear dependencies for every $n$, then we use Proposition \ref{mainInequality} to conclude that 
$X$ has \textit{structure}, in the sense that it has an infinite intersection with a translate of a subtorus of $\Tn$. 
This enables us to perform induction on $N$. Let us now introduce some preparatory lemmas.

\begin{lemma}\label{inductLem}
Let $\bA(x)\in \mM_{L\times N}(\Z[x])$ be of rank $\ell$ and satisfy condition (\ref{cond2}) from Theorem \ref{hoangMike}. Then there exist matrices $T\in \mM_{L\times\ell}(\Q)$, $\bB(x)\in \mM_{\ell\times L}(\Z[x])$ such that 
\begin{enumerate}[(i)]
\item $\bA(x)=T\bB(x)$,
\item $\bB_{\ast}(x)$ has full rank, and
\item There is a positive integer $q$ such that $qT$ is integral and  $\|qT\|_{\infty}\ll_{\ell} \|\bA_{\ast}\|_{\infty}^{\ell}$.
\end{enumerate}
\end{lemma}
\begin{proof}
Without loss of generality we my assume the first $\ell$ rows of $\bA_{\ast}(x)$ are $\Q$-linearly independent. Then there is an $L\times\ell$ matrix $T$ with entries in $\Q$ such that $\bA_{\ast}=T\bB_{\ast}$ where $\bB_{\ast}=\bB_{\ast}(x)\in \mM_{\ell\times N}(\Z[x])$ is the block of the first $\ell$ rows of $\bA_{\ast}(x)$. We claim that condition (b) guarantees that $A_{0}=TB_{0}$ for some $\ell\times N$ integral matrix $B_{0}$. First we show $\ker \left( T^t \right)\subset \ker\left(A_{0}^t\right)$.

Suppose $\bv\in\ker(T^t)$. Then $\bA^{t}_{\ast}\bv=\bB_{\ast}^{t}T^{t}\bv=0$, which implies $\bv\cdot \bA_{\ast}\bw=0$ for any $\bw\in\Qn$. But by condition (b) this implies that $\bv\cdot A_{0}\bw=0$ for each $\bw\in\Qn$, which implies $A_{0}^{t}\bv=0$. That is, $\bv\in\ker(A_{0}^t)$.

Therefore there exists $B_{0}\in M_{\ell\times N}(\Q)$ such that $A_{0}=TB_{0}$. But the uppermost $\ell\times\ell$ block of $T$ is the identity. Thus $B_{0}$ is none other than  the uppermost $\ell\times N$ block of $A_{0}$, and consequently $B_0$ is integral. Upon putting $\bB=\bB_{\ast}+ B_0$, we have $\bB$ is integral and $\bA=T\bB$.

Let $A$ be the $L\times DN$ matrix given by $A=[A_{1}\cdots A_{D}]$ and $B$ be the $\ell \times DN$ matrix given by $B=[B_{1}\cdots B_{D}]$. Since $A_{d}=TB_{d}$ for each $d=1,...,D$, we have $A=TB$. $B$ must have rank $\ell$ since $\bB_{\ast}(x)$ does, so there is an invertible $\ell\times\ell$ minor $B^{\prime}$ of $B$. Let $A^{\prime}$ be the corresponding minor of $A$ and observe we have the equality $A^{\prime}(B^{\prime})^{-1}=T$. Let $q=\det B^{\prime} \neq 0$ and $C=q^{-1}(B^{\prime})^{-1}$ be the adjugate of $B^{\prime}$. We then have the inequality 
\[
\|qT\|_{\infty}=\|A'C \|_{\infty} \ll_{\ell} \|A'\|_{\infty} \|C\|_{\infty} \ll_{\ell} \|\bA_{\ast}\|_{\infty}^{\ell}
\]
as required. Clearly we may assume $q$ to be positive.
\end{proof}

Our crucial tool is the following consequence of Proposition
\ref{mainInequality}. We regard it as some sort of \textit{inverse result} since it tells about the structure of $X$ if dilations of $X$ fail to be $\ep$-dense. 
In this respect our use of Proposition \ref{mainInequality} is rather different from Alon-Peres. It is perhaps no surprise that our proof of Proposition \ref{infinitePtsLemma} involves Ramsey's theorem.

\begin{proposition}\label{infinitePtsLemma}
Suppose $\ep>0$, $X$ is an infinite subset of $\Tn$, and $\bB(x)\in M_{\ell\times N}(\Z[x])$ such that $\bB_{\ast}(x)$ has full rank. If $\bB(r)X$ is not $\ep$-dense in $\T^{\ell}$ for any $r\in\Z$, then there exists a point $\by_{0}\in X$, an integer $J$, and  nonzero $\bw\in\Zn$ such that $\bw\cdot(\by-\by_{0})=J$ for infinitely many $\by\in X$. 
\end{proposition}
Note that the last equation is an equality in $\R$ rather than in $\T$, by our identification of points in $\Tn$ with their representatives in $[0,1)^{N}$.
\begin{proof}
We create a complete graph whose vertex set is $X$ and whose edges $(\bx,\by)$ are colored $\bw\in\Zn$ ($0<\|\bw\|_{\infty} \leq M\ell \|\bB_{\ast}\|_{\infty}$) if $\bw\cdot(\bx-\by)\in\Z$ and\footnote{Observe we are allowing multiple colors per edge.} colored $\bomega$ otherwise. By the infinite version of Ramsey's theorem there exists an infinite complete monochromatic subgraph whose vertex set is $Y\subset X$. We now would like to show that this graph cannot be $\bomega$-colored. 

Suppose, by way of contradiction, that the graph is $\bomega-$colored. For any distinct $\bx_{1},...,\bx_{k}$ in $Y$ and $R>0$ we have, by Proposition \ref{mainInequality}:
 				      \begin{eqnarray} \label{main3}
					    k^{2}
					    &\ll_{\ell}&  \dfrac{1}{\ep^{\ell}}\underset{\bm\in\Z^{\ell}}{\dsum_{0<\|\bm\|_{\infty}\leq M}} \dsum_{i=1}^{k}\dsum_{j=1}^{k} \lim_{R \rightarrow \infty} \dfrac{1}{R}\dsum_{r=1}^{R} e_{\bm}\big(\bB(r)(\bx_{i}-\bx_{j})\big) \nonumber \\ \nonumber
					    &=& \dfrac{1}{\ep^{\ell}}\underset{\bm\in\Z^{\ell}}{\dsum_{0<\|\bm\|_{\infty}\leq M}} \dsum_{i=1}^{k}\dsum_{j=1}^{k} \lim_{R \rightarrow \infty} \dfrac{1}{R}\dsum_{r=1}^{R} e_{\bm}\left(\dsum_{d=0}^{D}r^{d}B_{d}(\bx_{i}-\bx_{j})\right) \nonumber \\
					    &=& \dfrac{1}{\ep^{\ell}}\underset{\bm\in\Z^{\ell}}{\dsum_{0<\|\bm\|_{\infty}\leq M}} \dsum_{i=1}^{k}\dsum_{j=1}^{k} \lim_{R \rightarrow \infty} \dfrac{1}{R}\dsum_{r=1}^{R} e\left(\dsum_{d=1}^{D}r^{d}B_{d}^{t}\bm\cdot(\bx_{i}-\bx_{j})\right) \nonumber \\
					    &\ll_{\ell} & \dfrac{M^{\ell}}{\ep^{\ell}} \dsum_{i=1}^{k}\dsum_{j=1}^{k} \lim_{R \rightarrow \infty} \dfrac{1}{R}\dsum_{r=1}^{R} e\left(\dsum_{d=1}^{D}r^{d}B_{d}^{t}\bm \cdot (\bx_{i}-\bx_{j}) \right) 
				      \end{eqnarray}		
where $\bm$ is the lattice point which maximizes the last sum. Let $\tilde{d}$  be the largest index such that $B_{\tilde{d}}^{t}\bm\neq \bzero$. Then $\tilde{d}>0$ because $\bB^{t}_{\ast}(x)$ has $\Q$-linearly independent columns, which implies $B_{d}^{t}\bm$ is not zero for some $d=1,\ldots,D$.  For any $i \neq j$, since $(\bx_i, \bx_j)$ is $\bomega$-colored under our coloring and $\| B_{\tilde{d}}^{t}\bm \| \leq M\ell \|\bB_{\ast} \|_{\infty}$, we have
\begin{equation} \label{distinct}
B_{\tilde{d}}^{t}\bm \cdot(\bx_{i}-\bx_{j}) \neq 0
\end{equation}
Therefore, if $i \neq j$, the polynomial
\[
\Phi_{ij}(r)=\bm \cdot \bB_{\ast}(r) (\bx_{i}- \bx_{j}) = \dsum_{d=1}^{D} r^{d} B_{d}^{t} \bm \cdot (\bx_{i}-\bx_{j})
\]
has degree $\tilde{d}$. By Weyl's equidistribution theorem and Hua's bound (Lemma \ref{hua}), we have:
\begin{equation*}
\lim_{R\ra\infty}\dfrac{1}{R}\dsum_{r=1}^{R}e\left( \Phi_{ij}(r)   \right) = 
\begin{cases}
 0, & \textup{if }\Phi_{ij} \textup{ has at least one irrational coefficient} \\
 \ll_{D} b^{-1/\tilde{d}} \leq b^{-1/D}, & \textup{if } \Phi_{ij}(x)\in\Q[x], 
\end{cases}
\end{equation*}
where in the second case $b=b(i,j)$ is the least positive integer such that $b(\bm\cdot \bB_{\ast}(x)(\bx_{i}-\bx_{j}))\in\Z[x]$. 

For each $b>1$ we define 
		    \begin{eqnarray*}
			  S_{b}=\Big\lbrace  (i,j) 
			  &:&  1\leq i,j\leq k,\;\; b \text{ is the smallest positive integer} \\ && \text{such that}\;\; b(\bm\cdot \bB_{\ast}(x)(\bx_{i}-\bx_{j}))\in\Z[x]\Big\rbrace.
		    \end{eqnarray*}
Let $\ms_{b}=\# S_{b}$ and $\mS_{b}=\ms_{2}+\cdots+\ms_{b}$. Let $x_{i}=B_{\tilde{d}}^{t}\bm\cdot\bx_{i}$ for any $i=1, \ldots, k$, then the $x_i$ are distinct in $\T$ in view of (\ref{distinct}). We notice that if $(i,j)\in S_{b}$
then $b(x_i - x_j) \in \Z$. Consequently, $\mS_{b}\leq H_{b}$  where $H_{b}=h_{1}+\cdots+h_{b}$ and $h_{m}$ is the quantity defined by (\ref{hm}) for the sequence $x_1, \ldots, x_k$. We also have the trivial bound $\mS_{b} \leq k^2$ for any $b$, since for each couple $(i,j)$ we associate at most one $b$. Therefore
			    \begin{eqnarray*}
				  k^{2} \ll_{l,D} \dfrac{M^{\ell}}{\ep^{\ell}}\left( k + \dsum_{b=2}^{\infty} \ms_{b}b^{-1/D} \right)
			    \end{eqnarray*}
Combining this with Corollary \ref{alonPeresCorollary2} we have 
			  \begin{equation}
				k^{2} \ll_{D,\ep,\ell} k^{2-1/(2D)} 
			  \end{equation}
which is a contradiction. 

Therefore there is an infinite complete monochromatic subgraph whose color is $\bw$ for some $\bw\in \Zn$ and $0<\|\bw\|_{\infty}\leq M\ell \|B_{\ast} \|_{\infty}$. More specifically we find that there is an infinite subset $Y\subset X$ such that $\bw\cdot(\by-\by')\in\Z$ for any $\by,\by' \in Y$. Now fix an element $\by_{0}\in Y$. Upon noticing that the map $\by\mapsto \bw\cdot(\by-\by_{0})$ has a finite image (since $\by, \by_0 \in [0,1)^{N}$) and $Y$ is infinite, there exists an integer $J$ such that $\bw\cdot(\by-\by_{0})=J$ for infinitely many $\by\in Y$.
\end{proof}

We are now in a position to finish the proof of Theorem \ref{hoangMike}.

\begin{proof}[Proof of sufficiency of (\ref{cond1}) and (\ref{cond2})]
First we will provide a proof when $N=1$ and then proceed by induction on $N$.\newline

Let $X\subset \T$ be an infinite subset, $0<\ell\leq L$ be the rank of $\bA(x)$, and $\bB(x)$ and $T$ be given by Lemma \ref{inductLem}. 
We claim that for any $\ep>0$ there is an integer $n$ such that $\bB(n)X$ is $\ep$-dense in $\T^{\ell}$. 
Assume, by way of contradiction, that there exists an $\ep_{0}>0$ such that $\bB(n)X$ is not $\ep_{0}$-dense in $\T^{\ell}$ for any $n\in\Z$. 
By Proposition \ref{infinitePtsLemma}  there exists an integer $m \neq 0$, a point $y_{0}\in X$, an integer $J$ such that $m(y-y_{0})=J$ for infinitely many $y\in X$. 
This is clearly impossible (recall that this is an equality in $\R$). Therefore for every $\ep>0$ there exists an integer $n$ such that $\bB(n)X$ is $\ep$-dense in $\T^{\ell}$. Let $\mT=\mathrm{Im}(T)/\Zl$  where $\mathrm{Im}(T) \subset \Rl$ is the image of $T$.  Let $q$ be given by Lemma \ref{inductLem}. 
Then $qT$ is integral and well-defined when considered as a map from $\T^{\ell}$ to $\mT$. Letting $X/q=\lb \bx/q: \bx\in[0,1)^{N} \text{ and }\bx \in X\rb$ we find that $\bA(n)X=(qT)\bB(n)(X/q)$. Therefore for any $\ep>0$ there exists an integer $n$ such that $\bA(n)X$ is $\ep$-dense in $\mT$.

Now we assume the theorem holds for each integer up to $N-1$. Again, by Lemma \ref{inductLem} there exist an $L\times \ell$ matrix $T$ with entries in $\Q$, an $\ell\times N$ matrix $\bB=\bB(x)$ with entries in $\Z[x]$, a positive integer such that 
						\[
								\bA=T\bB
						\]
and the rows of  $\bB_{\ast}$ are $\Q$-linearly independent. Define
					\[
							X/q=\lb \bx/q: \bx\in[0,1)^{N} \text{ and }\bx \in X\rb.
					\]
and $\mT=\textrm{Im}(T)/\Z^{L}$, so that $qT$ is integral and well-defined as a map from $\T^{N-1}$ to $\mT$.  We have two possibilities:
\begin{enumerate}[(i)]
   \item Either for every $\ep>0$ there exists an integer $n$ such that $\bB(n)(X/q)$ is $\ep$-dense in $\T^{\ell}$. This implies that $\bA(n)X = (qT) \bB(n) (X/q) $ is $\tilde{\ep}$-dense in $\mT\subset{\Tl}$, where $\tilde{\ep} \ll \ep \|qT\|_{\infty}$.
   \item Or there exists an $\ep_{0}>0$ such that $\bB(n)(X/q)$ is not $\ep_{0}$-dense in $\T^{\ell}$ for any $n\in\Z$ . 
\end{enumerate}
If we are in the first case, then we are done. We suppose (ii), and rename $X/q$ as $X$. Proposition \ref{infinitePtsLemma} tells us that there is a nonzero $\bw\in\Zn$ and an infinite subset $Y\subset X$ such that $\by\mapsto \bw\cdot\by$ is constant on $Y$. We can assume $\bw\cdot\by=0$ for each $\by\in Y$ since this amounts to translating $X$ by a fixed $\btheta\in\Tn$. Let the subtorus $\mT$ of $\Tn$ be defined by $\mT=\lb \bt\in [0,1)^{N}: \bw\cdot\bt=0\rb$. Then there is an $N\times (N-1)$ matrix $H$ with full rank and integral entries such that 
\begin{equation} \label{h}
\mathrm{Im}(H) /\Z^{N}=\mT
\end{equation}
Since the mapping $\bt\mapsto H\bt+\Zn\in\mT$ is surjective, there is an infinite subset $Z\subset \T^{N-1}$ such that $H Z = Y$. 

Let $\bC(x)=\bA(x)H$, then $\bC$ is an $\ell\times (N-1)$ matrix. Let us verify that $\bC$ satisfies conditions (\ref{cond1}) and (\ref{cond2}). 
Suppose there is $\bq \in \Q^{N-1}$ such that $\bC_{\ast} \bq = \bzero$. Then $\bA_{\ast}(x) H \bq = \bzero$. Since $\bA$ satisfies (\ref{cond1}), it follows that $H \bq = \bzero$. Since $H$ has a trivial kernel, this implies that $\bq = \bzero$ and $\bC$ satisfies condition (\ref{cond1}). 
To see that $\bC$ satisfies condition (\ref{cond2}), let vectors $\bv \in \Q^{\ell}$ and $\bw\in\Q^{N-1}$ be such that $\bv\cdot\bC_{\ast}(x)\bw=0$ identically. Upon setting $\tilde{\bw}=H \bw\in\Q^{N}$, we find that $\bv\cdot \bA_{\ast}(x)\tilde{\bw}=0$ is the zero polynomial. Since $\bA(x)$ satisfies condition 
(\ref{cond2}), it follows that $0=\bv\cdot A_{0}\tilde{\bw}=\bv\cdot A_{0}H\bw=\bv\cdot \bC(0)\bw$. 

Let us now invoke the inductive hypothesis for $\bC$. It follows that there is a subtorus $\mT$ such that for every $\ep>0$ there exists $n$ such that $\bC(n)Z$ is $\ep$-dense in a translate of $\mT$. But $\bA(n)Y=\bC(n)Z$, so we are done.
\end{proof}

\begin{remarks}
It may not be clear from the proof why conditions (\ref{cond1}), (\ref{cond2}) are the correct ones. 
At first sight, it would seem that the only conditions we need in order to make the proof work are the weaker ones:
\begin{itemize}
 \item $T \neq 0$, which is equivalent to $\bA \neq 0$.
 \item $\mathrm{Ker}(T^t) \subset \mathrm{Ker}(A_{0}^t)$, which is equivalent to $\mathrm{Ker}(\bA_{\ast}^t) \subset \mathrm{Ker}(A_{0}^t)$.
\end{itemize}
But we want to maintain these requirements throughout our inductive process. Recall that our matrix $\bA$ is changed after each step, 
so keeping these requirements at each step ultimately leads to conditions (\ref{cond1}) and (\ref{cond2}).
\end{remarks}

\section{The finite version} \label{finite}

In order to make the proof of Theorem \ref{hoangMike} effective, we need to keep track of all the quantities involved when we move from one dimension to the next. The main obstacle in the proof of Theorem \ref{polynomialEffective} is finding an effective version of Proposition \ref{infinitePtsLemma}. One could use the finite version of Ramsey's theorem, 
but currently we don't have a sensible bound for Ramsey numbers which involve more than two colors. We can get past this, by noticing that the graph we used in 
Proposition \ref{infinitePtsLemma} is a very special graph. The following lemma is an effective form of Proposition \ref{infinitePtsLemma}.

\begin{proposition}\label{effectiveLemma}
		Let $\bB(x)\in\mM_{\ell\times N}(\Z[x])$ have full rank and let $X=\lb \bx_{1},...,\bx_{k}\rb\subset \Tn$ be a set of $k$ distinct points. If $\bB(n)X$ is not $\ep$-dense in $\T^{\ell}$ for any $n=1,2,...$ then there exists a subset $Y\subset X$, $\by_{0}\in X$, $\bw\in\Zn$, and $J\in\Z$ such that 
 \begin{eqnarray}
			& &	       \bw\cdot(\by-\by_{0})=J \;\;\textup{ for each }\by\in Y \label{y}, \\
			& &	       \|\bw\|_{\infty}\ll_{\ell,N}  \|\bB_{\ast}\|_{\infty}\ep^{-1} \label{j}, \text{ and} \\
			& &	     \ep^{\ell+1}k^{1/4D} \| \bB_{\ast} \|_{\infty}^{-1} \ll_{\ell,N,D} |Y|. \label{k}
\end{eqnarray}
\end{proposition}
Note that again, (\ref{y}) is an equality in $\R$.

	   \begin{proof}
 By Proposition \ref{mainInequality} we have a constant $M \ll_{\ell} \ep^{-1}$ such that 
	      \begin{equation}
		     k^{2}\ll_{\ell} \dfrac{1}{\ep^{\ell}} \underset{\bm\in\Z^\ell}{\dsum_{0<\|\bm\|_{\infty}\leq M}} \dsum_{\bx\in X}\dsum_{\by\in X} \lim_{R \rightarrow \infty} \dfrac{1}{R}\dsum_{r=1}^{R} e\left(\dsum_{d=0}^{D}r^{d}B_{d}^{t}\bm\cdot(\bx-\by)\right)
	      \end{equation}
where $e(t)=\exp(2\pi i t)$ and $M\ll_{\ell}\ep^{-2}$. By an abuse of notation, let $\bm\in\Z^{\ell}$ (with $0<\|\bm\|_{\infty}\leq M$) be the lattice point which maximizes the first sum. Then
\begin{equation} \label{mainIneqEffPoly}
 k^2 \ll_{\ell} \dfrac{M^{\ell}}{\ep^{\ell}} \dsum_{\bx\in X}\dsum_{\by\in X} \omega(\bx,\by)
\end{equation}
where $\omega(\bx,\by)$ is the weight given by
		\begin{equation*}
		      \omega(\bx,\by)=\left|\ds\lim_{R\ra\infty}  \dfrac{1}{R}\dsum_{r=1}^{R} e\left(\dsum_{d=1}^{D}r^{d}B_{d}^{t}\bm(\bx-\by)\right) \right|
		\end{equation*}  
Let $d$ be the largest integer such that $B_{d}^{t}\bm\neq0$, then $d>1$ since $\bB_{\ast}$ has full rank. We  partition 
$X$ into equivalence classes $R_{1},...,R_{s}$, with $|R_{i}|=c_{i}$, where $\bx \sim \by$ if $B_{d}^{t}\bm\cdot(\bx-\by)\in\Z$. 
      
Define 
\[
\Phi_{i,j}(r)=\bm\cdot\bB_{\ast}(r)(\bx_{i}-\bx_{j}) = \dsum_{d=1}^{D}r^{d}B_{d}^{t}\bm(\bx-\by)
\]
then $\Phi$ has degree $d$. We use Weyl's equidistribution theorem and Hua's bound to obtain
		\begin{equation} \label{huabound}
		      \omega(\bx_{i},\bx_{j})\leq \begin{cases}
			      1 & \text{ if } \bx\sim \by \\
			      b^{-1/d} & \text{ if }\bx\not\sim \by\;\;\text{ and }\Phi_{ij}(x)\in\Q[x] \\
			      0 &\text{ if } \Phi_{ij} \text{ has at least one irrational coefficient}.
			  \end{cases}
		\end{equation}
where in the second case $b=b(i,j)$ is the smallest positive integer such that $b \Phi_{ij}(x) \in \Z[x]$.

Let $y_{1},...,y_{s}\in\T$ be given by $y_{i}=B_{d}^{t}\bm\cdot \bx_{i}$ for some $\bx_{i}\in R_{i}$. Then by the way we define equivalence classes, $y_{1},...,y_{s}$ are distinct in $\T$. By substituting the bound (\ref{huabound}) into (\ref{mainIneqEffPoly}), we have:

\begin{eqnarray*}
k^{2} &\ll_{\ell} &
		      \left(\dfrac{M}{\ep}\right)^{\ell} \dsum_{i=1}^{s} \dsum_{j=1}^{s} \dsum_{\bx_{i}\in R_{i}}\dsum_{\bx_{j}\in R_{j}} \omega(\bx_{i},\bx_{j}) \\
		      &\leq & \left(\dfrac{M}{\ep}\right)^{\ell}\lb\dsum_{i=1}^{s}c_{i}^{2} +\underset{i\neq j}{\dsum_{i=1}^{s}\dsum_{j=1}^{s}}\dsum_{\bx_{i}\in R_{i}}\dsum_{\bx_{j}\in R_{j}}\omega(\bx_{i},\bx_{j})\rb\\
		       &\leq &\left(\dfrac{M}{\ep}\right)^{\ell}\lb \dsum_{i=1}^{s}c_{i}^{2}+c^{2}\dsum_{b=2}^{\infty}\ms_{b}b^{-1/d}\rb
\end{eqnarray*}
		where 
		\begin{eqnarray*}
		\ms_{b}=\#\{ (i,j) &:& 1\leq i,j\leq s,\;\;  b \text{ is the smallest positive integer} \\ & & \text{ such that }  \; b\Phi_{ij}(x)\in\Z[x] \}
		\end{eqnarray*}
and $c=\max\lb c_{1},...,c_{s}\rb$. Clearly the sequence $\ms_b$ satisfies the conditions of Corollary \ref{alonPeresCorollary2}. Upon writing $c_{1}+\cdots+c_{s}=k$
and noticing $s \leq k$, we have 
		\begin{equation*}
		      k^{2}\ll_{D,\ell} \left(\dfrac{1}{\ep}\right)^{\ell}\lb kc +c^{2}s^{2-1/(2D)} \rb \ll_{D,\ell} \ep^{-2 \ell} c^{2}k^{2-1/(2D)}.
		\end{equation*}
		That is,
\begin{equation*}
					\ep^{\ell}k^{1/4D}\ll_{\ell, D}c. 					
		\end{equation*}
Now let  $Y'$ be equal to one of the equivalence classes $R_{1},...,R_{s}$ whose cardinality is $c$, and $\bw=B_{d}^{t}\bm$. Then $\bw\cdot(\bx-\by)\in\Z$ for each $\bx,\by\in Y^{\prime}$. But seeing that $|\bw\cdot(\bx-\by)|\leq N\|\bw\|_{\infty}$, we are guaranteed the existence of an integer $|J|\leq N\|\bw\|_{\infty}$ and $\by_{0}\in Y^{\prime}$ such that $\bw\cdot(\by-\by_{0})=J$ for at least $c/N\|\bw\|_{\infty}$ elements $\by$ of $Y^{\prime}$. But 
      		\[\|\bw\|_{\infty}\ll_{N,\ell} \|\bB_{\ast}\|_{\infty}M\ll_{N, \ell} \|\bB_{\ast}\|_{\infty}\ep^{-1}\]
      	Combining this with the above we have the existence of a subset $Y\subset Y^{\prime}\subset X$ such that 
      			\begin{equation*}
      					\ep^{\ell+1}k^{1/4D} \|\bB_{\ast} \|_{\infty}^{-1} \ll_{\ell ,N, D} |Y|
      			\end{equation*}
as desired.
	   \end{proof}
We also need to estimate the entries of the matrix $H$ introduced in (\ref{h}). 
\begin{lemma}\label{hProp}
					 Let $\bw\in\Zn$ be nonzero and $\bw^{\perp}=\lb \bv\in\Rn: \bv\cdot\bw=0\rb$. There exists an $(N-1)\times N$ integral matrix $H$ whose image is $\bw^{\perp}$ and $\|H\|_{\infty}= \|\bw\|_{\infty}$. 
\end{lemma}
\begin{proof}
						Since $\bw=(w_{1},...,w_{N})$ is nonzero we may assume without loss of generality that  $w_{N}\neq0$. Let 
									\[   \bv_{j}=w_{N}\be_{j}-w_{j}\be_{N}. \]
where $(\be_1, \ldots, \be_N)$ is the standard basis of $\R^N$. Then $\bv_{j}\in\bw^{\perp }$ because 
						\[\bv_{j}\cdot\bw=w_{N}\be_{j}\cdot\bw-w_{j}\be_{N}\cdot\bw=0.\]
						Clearly $\bv_{1},...,\bv_{N-1}$ are linearly independent and therefore form a basis for $\bw^{\perp}$. Letting $H$ be the $N\times (N-1)$ matrix whose columns are $\bv_{1},...,\bv_{N-1}$ gives the result.
\end{proof}
We are now in a position to prove Theorem \ref{polynomialEffective}.
\begin{proof}[Proof of Theorem \ref{polynomialEffective}] Let us proceed by induction. \newline

\noindent \textbf{Base case:}
Let $N=1$ and $\bA(x)$ be an $L\times 1$ matrix with entries in $\Z[x]$, having rank $\ell$, degree at most $D$, and satisfy conditions (\ref{cond1}) and (\ref{cond2}) of Theorem \ref{hoangMike}. Let $X=\lb x_{1},...,x_{k}\rb$ be a set of $k$ distinct points in $\T$ such that there does not exist a subtorus $\mT$ such that $\bA(n)X$ is not $\ep$-dense in a translate of $\mT$ for any $n=1,2,\ldots$. 

By Lemma \ref{inductLem}, there exist an $\ell\times N$ matrix $\bB(x)$ whose rows are rows of $\bA(x)$, an $L\times \ell$ matrix $T$ with entries in $\Q$ such that  $\bB_{\ast}(x)$ has full rank and $\bA(x)=T\bB(x)$. Furthermore, there is a positive integer $q$ such that $qT$ is integral and  $\|qT\|_{\infty}\ll_{\ell} \|\bA_{\ast}\|_{\infty}^{\ell}$. Define 
\[
X/q=\lb x/q+\Z:x\in [0,1)\;\text{ and }x \in X \rb
\]
then $X/q$ also has cardinality $k$, and $(qT)\bB(n)(X/q)=\bA(n)X$ is not $\ep$-dense in any translate of $\mT=\mathrm{Im}(T)/\Zl$. This implies that $\bB(n)(X/q)$ is not $\ep_{1}$ dense in $\T^{\ell}$ for any $n=1,2,\ldots$, where $\ep_1 \gg _{L} \ep/\|qT\|_{\infty}$. Therefore by Proposition \ref{effectiveLemma}, there exists a subset $Y\subset X/q$, 
$y_0 \in \T$, integers $J$ and $w$ such that 
				  \begin{equation} \label{j2}
				       w(y-y_0)=J \;\;\text{ for each }y\in Y, 
				    \end{equation} 
				    \begin{equation} \label{y2}
					  \ep_{1}^{\ell+1}k^{1/4D} \|\bB_{\ast} \|_{\infty}^{-1} \ll_{L,D} |Y|,
				    \end{equation}
But (\ref{j2}) cannot happen for more than one value of $y$ (recall that it's an equality in $\R$), Combining this with (\ref{y2}), we have
\begin{equation}
					k\ll_{L,D} \|\bB_{\ast}\|_{\infty}^{4D}\left( \dfrac{1}{\ep_{1}}\right)^{4D(\ell+1)} \leq \|\bB_{\ast}\|_{\infty}^{4D}\left( \dfrac{1}{\ep_{1}}\right)^{4D(L+1)}
\end{equation}
      Recall that $\ep_{1} \gg_{L} \ep/\|qT\| \gg_{L} \ep \| \bA_{\ast} \|_{\infty}^{-\ell} \geq \ep \| \bA_{\ast} \|_{\infty}^{-L}$. We also trivially have $\|\bB_{\ast}\|_{\infty}\leq \|\bA_{\ast}\|_{\infty}$ (since the rows of $\bB$ are the rows of $\bA$ by construction) so 
				\begin{equation}
					k\ll_{L,D} \|\bA_{\ast}\|_{\infty}^{4D(L(L+1)+1)} \left( \dfrac{1}{\ep}\right)^{4D(L+1)}
				\end{equation}
which shows that $k(\ep;L,1,\bA)$ exists and can be bounded by the right hand side.\\
\\
\noindent \textbf{Inductive step.} Now we assume that for each $\bC\in \mM_{L\times n}(\Z[x])$ having degree $D$ and that satisfies conditions (\ref{cond1}) and (\ref{cond2}) of Theorem \ref{hoangMike},  there exist constants $c_{1}(n,L,D)$ and $c_{2}(n,L,D)$ such that 
						\begin{equation}
							k(\ep;L,n,\bC)\ll_{N,L,D} \|\bC_{\ast}\|_{\infty}^{c_{1}(n,L,D)}\left( \dfrac{1}{\ep} \right)^{c_{2}(n,L,D)}.
						\end{equation}
for $n=1,2,...,N-1$. 

Let $\bA(x) \in \mM_{L\times N}(\Z[x])$ have degree at most $D$ and satisfy conditions (\ref{cond1}) and (\ref{cond2}) from Theorem \ref{hoangMike}. Suppose that $X=\lb \bx_{1},...,\bx_{k}\rb$ is a set of $k$ distinct points in $\Tn$ such that there does not exist a subtorus $\mT$ of $\T^{L}$ such that $\bA(n)X$ is $\ep$-dense in a translate of $\mT$ for any $n=1,2,...$. Suppose $\bA(x)$ has rank $\ell$. Again, let $\bB(x)\in\mM_{\ell\times N}(\Z[x]), \;T\in\mM_{L\times \ell}(\Q)$ and $q \in \Z$
be given by Proposition \ref{inductLem}, and let $X/q=\lb\bx/q:\bx\in[0,1)^{N}\text{ and }x\in X\rb$. As before we see that $\bB(n)(X/q)$ cannot be $\ep_{1} \gg \ep/\|qT\|-$dense in $\Tl$ for any $n=1,2, \ldots$. Therefore by Lemma \ref{effectiveLemma} then there exists a subset $Y\subset X/q$, $\by_{0}\in\Tn$, $J\in\Z$ and a $\bw\in\Zn$ such that 
				    \begin{equation}
				       \bw\cdot(\by-\by_{0})=J \;\;\text{ for each }\by\in Y, 
				    \end{equation} 
				    \begin{equation}\label{y3}
					  \ep_{1}^{\ell+1}k^{1/(4D)} \|\bB_{\ast} \|^{-1}_{\infty} \ll_{N,L,D} |Y|, \text{ and }
				    \end{equation}
				    \begin{equation}
					    0< \|\bw\|_{\infty}\ll_{L,N} \|\bB_{\ast}\|_{\infty}\ep_{1}^{-1}.
				    \end{equation}
Clearly $Y$ lies in a translate of the torus $\mT=\lb \bx+\Zn: \bx \in [0,1)^{N}, \bx\cdot\bw=0 \rb\subset \Tn$. By Lemma \ref{hProp}, there is a matrix $H\in\mM_{N\times (N-1)}(\Z)$ of rank $N-1$ such that the range of $H$ is $\bw^{\perp}$ and $\| H\|_{\infty}=\|\bw\|_{\infty}$.  $H$ is surjective as a map from $\T^{N-1}$ to $\mT$ so there is a set $Z$ of cardinality $|Z|=|Y|$ points in $\T^{N-1}$ such that $HZ=Y$. By the definition of the function $k(\ep;L,N,\bA)$, we have that
		\begin{equation}
				|Y|=|Z| \leq k(\ep_1;L,N-1,\bA H).
		\end{equation}

Note that the degree of $\bA H$ is at most $D$, so by the inductive hypothesis and (\ref{y3}) we have
\begin{equation}\label{eq:4}
					  \ep_{1}^{L+1}k^{1/(4D)} \| \bB_{\ast} \|_{\infty}^{-1} \ll_{N,L,D} \|(\bA H)_{\ast}\|_{\infty}^{c_{1}(N-1,L,D)}\left( \dfrac{1}{\ep_1} \right)^{c_{2}(N-1,L,D)}
\end{equation}
But \[\|(\bA H)_{\ast}\|_{\infty}\ll_{N,L}\|\bA_{\ast}\|_{\infty}\|H\|_{\infty}=\|\bA_{\ast}\|_{\infty}\|\bw\|_{\infty}\ll \|\bA_{\ast}\|_{\infty}\ep_{1}^{-1}\]
and $\|\bB_{\ast} \|_{\infty} \leq  \|\bA_{\ast} \|_{\infty}$. Therefore,
\[
k^{1/(4D)} \ll_{N,L,D} \|\bA_{\ast}\|_{\infty}^{1+c_{1}(N-1,L,D)}\left( \dfrac{1}{\ep_1} \right)^{c_{1}(N-1,L,D)+c_{2}(N-1,L,D)+L+1} \\ 
\]
Recalling that $\ep_1 \gg_{N,L} \ep \|qT\|_{\infty}^{-1} \gg \ep \| \bA_{\ast}\|_{\infty}^{-L}$, we have
\begin{equation}
					  k \ll_{N,L,D} \| \bA_{\ast}\|_{\infty}^{c_{1}(N,L,D)}\left( \dfrac{1}{\ep} \right)^{c_{2}(N,L,D)}
\end{equation}
where
\[
c_{2}(N,L,D)= 4D \Big( c_{1}(N-1,L,D)+c_{2}(N-1,L,D)+L+1 \Big)   
\]
and
\[
c_1(N,L,D) = Lc_2 (N,L,D) + 4D \Big( 1+c_1(N-1,L,D) \Big)
\]
This shows that $k(\ep;L,N,\bA)$ exists, and establishes a bound of the desired form for $k(\ep;L,n,\bA)$.
\end{proof}
\begin{remark}
As we noted in the introduction, we do not attempt to find the optimal values of the exponents $c_{1}$ and $c_{2}$ and the values that we achieve can be improved. We found in the base step that $c_{1}(1,L,D)=4D(L(L+1)+1)$ and $c_{2}(1,L,D)=4D(L+1)$. It is not difficult to show that $c_{1}(N,L,D)\leq (CD)^{N}L^{N+1}$ and $c_{2}(N,L,D)\leq (CDL)^{N}$  for  $N,D,L\geq 1$, and $C$ is a positive constant with $C\leq 20$.   It would be interesting to know the true order of magnitude for the optimal exponents, even for fixed values of $N,L$, and $D$. When $N\geq L$ and $X=X_{m}^N$ where $X_{m}$ is the Farey sequence of order $m=2/\ep$, no dilation $n\mathbb{P} X$, where $\mathbb{P}$ is projection onto the first $L$ components, contains a point in the cube $(0,\ep)^{L}$. But $\# X=\Omega(\ep^{-2N})$ which implies that the optimal choice for $c_{2}(N,L,1)$  is at least $2N$ when $N\geq L$. This is how the lower bound for $k$ is obtained in \cite{MR1200973} when $N=L=1$ and it is \textit{nearly} sharp in this case. 
\end{remark}

\section{The High Dimensional Glasner Theorem}\label{glasnerSection}

In this section we prove a stronger result than Theorem \ref{HDGlasner}. The proof of Theorem \ref{HDGlasner} follows along the same lines of the proof of \cite[Proposition 6.1]{MR1143662}. 
Without any extra effort effort, we can add the extra requirement that the entries of $T$ be relatively prime. This is reminiscent of Theorem \ref{alonPeres} (\ref{ap1}) though perhaps any resemblance stops here. We have the following:

\begin{theorem}\label{primitiveGlasner}
				For any $\ep>0$ and any subset $X\subset \Tn$ of cardinality at least
								$k\gg_{L}\ep^{-3LN}$
				 there exists a matrix $T\in\mM_{L\times N}(\Z)$ with relatively prime entries such that $TX$ is $\ep$-dense in $\Tl$.
\end{theorem}
We note that the exponents we obtain can be easily improved, but we opt for cruder bounds for the sake of brevity.
 \begin{proof} 
Let $\ep>0$ and Let $X\subset \Tn$ have cardinality $k$ and let $X_{j}\subset \T$ be the projection of $X$ onto the $j^{th}$ coordinate axis for $j=1,2,...,N$. The projection homomorphism $\mathbb{P}_{j}$ is represented by inner product with the vector $(0,...,1,...,0)$ where the 1 is in the $j^{th}$ entry. Clearly
	  \begin{equation}
		  k=\# X\leq \dprod_{j=1}^{N}\#X_{j}.
	  \end{equation}
 Consequently there is a projection $X_{i}$ for which $\#X_{i}\geq k^{1/N}$. Let $Y$ be a subset of $X$ such that its projection on the $i^{th}$ coordinate 
 $Y_i \subset \T$ has cardinality at least $K=\lceil k^{1/N} \rceil$.  Now if we can find a primitive vector $\ba\in\Zl$ such that $\ba Y_i$ is $\ep$-dense in $\Tl$ we are done once setting $T$ equal to the composition of $\mathbb{P}_{i}$ and the homomorphism induced by multiplication by $\ba$. We will show that we can choose $\ba$ to be of the following form
	      \[
		    \ba=\ba(n)=(q_1n,q_{2}n+1,q_{3}n,...,q_{L}n)
	      \]
where we choose $q_{\ell}=(M+1)^{\ell-1}$ for $n\geq 1$ where $M=[L/\ep]$. Note that $\ba$ is primitive since $(n,q_{2}n+1)=1$. 

Suppose, by way of contradiction, that there is no $n$ for which $\ba Y=\ba(n)Y$ is $\ep$-dense in $\Tl$. Then we have by Proposition \ref{mainInequality} 
\begin{equation}\label{main5}
    					K^{2}\ll_{L} \dfrac{1}{\ep^{L}}\underset{\bm\in\Zl}{\dsum_{0<\|\bm\|_{\infty}\leq M}} \dsum_{x\in Y_{i}}\dsum_{y\in Y_{i}} \lim_{R \rightarrow \infty} \dfrac{1}{R}\dsum_{r=1}^{R}e\Big(\bm\cdot \ba(r)(x-y)\Big).
\end{equation}
By abuse of notation, let $\bm$ be the lattice point which maximizes the first sum. Then
\[
 K^{2}\ll_{L} \dfrac{M^L}{\ep^{L}} \dsum_{x\in Y_{i}}\dsum_{y\in Y_{i}} \lim_{R \rightarrow \infty} \dfrac{1}{R}\dsum_{r=1}^{R}e\Big(\bm\cdot \ba(r)(x-y)\Big).
\]

But 
    			\begin{eqnarray*}
    					\ds\lim_{R\ra\infty}\dfrac{1}{R}\dsum_{r=1}^{R}e\Big(\bm\cdot \ba(r)(x-y)\Big)
    					&=& \ds\lim_{R\ra\infty}\dfrac{1}{R}\dsum_{r=1}^{R}e\left(r(x-y)\dsum_{\ell=1}^{L} m_{\ell}q_{\ell} \right) \\
    					&=& \begin{cases}
    									1 & \text{ if } (x-y)\dsum_{\ell=1}^{L} m_{\ell}q_{\ell}\in\Z \\
    									0 & \text{ otherwise}.
    							\end{cases}
    			\end{eqnarray*}
Hence, 
\[
K^{2 }\ll_{L}\ep^{-2L} \#\{ (x,y): x, y \in Y_{i}, \; Q(x-y) \in \Z \}
\]
where $Q=\dsum_{\ell=1}^{L} m_{\ell}q_{\ell}$. Our choices of $q_{1},...,q_{L}$ guarantee that $Q$ is non-zero. The right hand side of the above inequality can be trivially be bounded (by the same reasoning as in Proposition \ref{count}) by 
\[\ep^{-2L} KQ \ll \ep^{-2L}K M^{L} \ll \ep^{-3L} K\] 
Recalling $K=\lceil k^{1/N} \rceil$ gives
    			\[
						    	k\ll_{L}\ep^{-3LN}.
    			\]
\end{proof}

\section{Concluding Remarks} \label{applications}

We conclude with a few remarks concerning our main results. For example, it is obvious by Theorem \ref{HDGlasner} that if $X\subset \Tn$ is an infinite subset then the union $\cup_{T}TX$  over all $T\in\mM_{L\times N}(\Z)$ is dense in $\Tl$. Moreover, if $X$ is invariant under the action of $\mM_{L\times N}(\Z)$, then $X$ is dense in $\Tl$.  Similarly, a simple compactness argument implies the following corollary Theorem \ref{hoangMike}.
\begin{corollary} \label{largeInvariant}
	Let $\bA(x)\in \mM_{L\times N}(\Z[x])$ satisfy conditions (\ref{cond1}) and (\ref{cond2}) of Theorem  \ref{hoangMike}. If $X\subset \Tn$ is an infinite subset, then the closure of $\cup_{n}\bA(n)X$ contains a translate of a subtorus $\mT$.  
\end{corollary}
In particular, if  $X$ is infinite and $X\subset\bA(n)X$ for each $n$, then the closure of $X$ contains a translate of a subtorus $\mT$. \newline
\indent  It would be interesting to see what kind of generalizations can be made of Theorem \ref{HDGlasner}. That is, what conditions on an infinite topological group $G_{1}$ and a metric group $G_{2}$ guarantee that for any infinite subset $X\subset G_{1}$, and $\ep>0$, there exists a continuous homomorphism $\varphi:G_{1}\ra G_{2}$ such that $\varphi(X)$ is $\ep$-dense in $G_{2}$? An interesting special case of this question occurs when $G_{1}$ is a compact (or locally compact) Abelian group and $G_{2}=U(1)=\lb z\in\C: |z|=1\rb$, the problem is to find a unitary character $\varphi$ of $G_{1}$ which distributes a prescribed set of points evenly throughout $U(1)$. \newline
\indent One necessary condition on $G_{1}$ is that for each $\ep>0$ there must exist a characters $\varphi$ for which $\varphi(G_{1})$ is $\ep$-dense in $U(1)$. Even though this condition is inherently necessary, it cannot be dismissed as a triviality. For instance,  if $G_1= \mathbf{F}_2^{\infty}$ with the metric $d(x,y)=\sum_{i=1}^{\infty}\frac{|x_i-y_i|}{2^i}$, then the group of all (continuous) characters of $G_1$ is $\mathbf{F}_2^{\omega} = \{x=(x_1,x_2 \ldots): x_i \neq 0 \textup{ for finitely many } i\}$ via
$x(y)= (-1)^{x \cdot y}$ for all $x \in \mathbf{F}_2^{\omega}, y \in \mathbf{F}_2^{\infty}$ (note that the dot product is well defined). 
But the image of the whole of $G_1$ under any $x$ is the set $\{-1,1\}$ and can't be $\ep$-dense.  \newline
\indent As noted in the introduction, Alon and Peres are able to estimate the discrepancy of dilations of the form $nX$ using the probabilistic method (see Theorem 1.2 from \cite{MR1143662}). It would be interesting to see an analogous result in higher dimensions.\newline
\indent  Baker \cite{MR2803785} has proven a quantitative lemma about dilations of the form $nX$ where $X\subset \Tn$, though his hypotheses and conclusion differ from our results. His proof makes use of Lemma \ref{multi-montgomery} as well.

\bibliographystyle{plain}
\bibliography{uniform}

\end{document}